\documentclass[12pt,reqno]{amsart}

\usepackage{mathdots}

\usepackage{color}

\usepackage{pdfsync}

\usepackage{enumitem}

\usepackage{wasysym}

\usepackage{amssymb}

\usepackage{mathrsfs}

\DeclareMathAlphabet{\mathpzc}{OT1}{pzc}{m}{it}

\usepackage[all]{xy}

\usepackage{tikz}
\usetikzlibrary{arrows,matrix,decorations.pathmorphing,decorations.pathreplacing,positioning,shapes.geometric,shapes.misc,decorations.markings,decorations.fractals,calc,patterns}

\usepackage{graphicx}

\usepackage{float}

\usepackage[bottom]{footmisc}

\usepackage{moreenum}

\usepackage{makecell}

\entrymodifiers={+!!<0pt,\fontdimen22\textfont2>}

\def\cF{\mathscr{F}}

\def\cO{\mathscr{O}}
\def\cP{\mathscr{P}}

\def\BC{\mathbb{C}}

\def\BH{\mathbb{H}}

\def\BN{\mathbb{N}}

\def\BQ{\mathbb{Q}}
\def\BR{\mathbb{R}}

\def\BZ{\mathbb{Z}}

\def\adots{\mathinner{\mkern1mu\raise1.0pt\vbox{\kern7.0pt\hbox{.}}\mkern2mu\raise4.0pt\hbox{.}\mkern2mu\raise7.0pt\hbox{.}\mkern1mu}}

\def\ast{{\textstyle *}}

\def\dddots{\mathinner{\mkern1mu\raise10.0pt\vbox{\kern7.0pt\hbox{.}}\mkern2mu\raise5.3pt\hbox{.}\mkern2mu\raise1.0pt\hbox{.}\mkern1mu}}
\def\dddotssmall{\mathinner{\mkern1mu\raise7.0pt\vbox{\kern7.0pt\hbox{.}}\mkern-1mu\raise4pt\hbox{.}\mkern-1mu\raise1.0pt\hbox{.}\mkern1mu}}

\def\id{\operatorname{id}}
\def\Image{\operatorname{Im}}

\def\Mob{\operatorname{M\ddot{o}b}}

\def\PSL2{\operatorname{PSL}_2}

\def\SL2{\operatorname{SL}_2}

\numberwithin{equation}{section}

\newtheorem{Lemma}{Lemma}[section]
\newtheorem{Theorem}[Lemma]{Theorem}

\theoremstyle{definition}
\newtheorem{Definition}[Lemma]{Definition}
\newtheorem{Setup}[Lemma]{Setup}

\newtheorem{Construction}[Lemma]{Construction}
\newtheorem{Remark}[Lemma]{Remark}

\newtheorem*{bfhpg*}{}

  {\begin{list}{}{%
    \settowidth{\labelwidth}{\textbf{#1:}}%
    \setlength{\leftmargin}{\labelwidth}\addtolength{\leftmargin}{\labelsep}}}%
  {\end{list}}

\begin{document}

\setlength{\parindent}{0pt}
\setlength{\parskip}{7pt}

\title[$p$-angulated Conway--Coxeter]{A $p$-angulated generalisation of Conway and Coxeter's theorem on frieze patterns}

\author{Thorsten Holm}
\address{Institut f\"{u}r Algebra, Zahlentheorie und Diskrete
  Mathematik, Fa\-kul\-t\"{a}t f\"{u}r Ma\-the\-ma\-tik und Physik, Leibniz
  Universit\"{a}t Hannover, Welfengarten 1, 30167 Hannover, Germany}
\email{holm@math.uni-hannover.de}
\urladdr{http://www.iazd.uni-hannover.de/\~{ }tholm}

\author{Peter J\o rgensen}
\address{School of Mathematics, Statistics and Physics,
Newcastle University, Newcastle upon Tyne NE1 7RU, United Kingdom}
\email{peter.jorgensen@ncl.ac.uk}
\urladdr{http://www.staff.ncl.ac.uk/peter.jorgensen}





\begin{abstract} 

Coxeter defined the notion of frieze pattern, and Conway and Coxeter proved that triangulations of polygons are in bijection with integral frieze patterns.  We show a $p$-angulated generalisation involving non-integral frieze patterns.  We also show that polygon dissections give rise to even more general non-integral frieze patterns.

\end{abstract}

\maketitle

\setcounter{section}{-1}
\section{Introduction}
\label{sec:introduction}

Recall the following definition due to Coxeter \cite[sec.\ 1]{C}. 

\begin{Definition}
\label{def:frieze_pattern}
Let $n \geqslant 0$ be an integer.  A {\em frieze pattern of width $n$} consists of $n+4$ infinite horizontal rows of non-negative real numbers, with an offset between even and odd rows, see Figure~\ref{fig:sqrt2_frieze} on page~\pageref{fig:sqrt2_frieze} and Figure~\ref{fig:sqrt3_frieze} on page~\pageref{fig:sqrt3_frieze}.  The rows are numbered $0$ through $n+3$, starting from below, and they satisfy:
\begin{enumerate}
\setlength\itemsep{4pt}

  \item  Rows number $0$ and $n+3$ consist of zeroes.  Rows number $1$ and $n+2$ consist of ones.  Rows number $2$ through $n+1$ consist of positive real numbers.
  
  \item  Each ``diamond'' 
$
\vcenter{
  \xymatrix @-2.1pc {
    & b & \\ a & & d \\ & c \lefteqn{\phantom{\widetilde{b}}} & 
                    }
        }
$
satisfies $ad - bc = 1$.
\end{enumerate}
Row number $2$ (the first non-trivial row from below) is called the {\em quiddity row}.
\hfill $\Box$
\end{Definition}

Some authors permit more general entries than we do, e.g.\ allowing rows $1$ through $n+2$ to contain zeroes.  However, our definition implies that rows $1$ through $n+2$ consist of non-zero numbers, and hence the quiddity row determines the whole frieze pattern by repeated applications of (ii).

Frieze patterns are deep objects, as one can see by replacing the positive real numbers by Laurent polynomials in the definition.  Then the entries of the frieze patterns become cluster variables in cluster algebras of type $A_n$, so frieze patterns preempted a special case of cluster algebras by $30$ years.  Frieze patterns have been the object of intensive research, which has revealed that they form a nexus between combinatorics, geometry, and representation theory; see \cite{M} for a recent survey.

\begin{Definition}
\label{def:type_Lambda_q_frieze_pattern}
Let $p \geqslant 3$ be an integer.  A frieze pattern is of {\em type $\Lambda_p$} if the quiddity row consists of (necessarily positive) integral multiples of
\[
\tag*{$\Box$}
  \lambda_p = 2\cos\Big( \frac{ \pi }{ p } \Big).
\]
\end{Definition}

For instance, $\lambda_3 = 1$ so type $\Lambda_3$ means that the quiddity row consists of (necessarily positive) integers.  By \cite[eq.\ (6.6)]{C} this is equivalent to the whole frieze pattern consisting of (necessarily positive) integers, except the zeroes in rows $0$ and $n+3$.  Hence the frieze patterns of type $\Lambda_3$ are precisely the integral frieze patterns considered by Conway and Coxeter in \cite{CC1} and \cite{CC2}.  In general, $\lambda_p$ is an algebraic integer, see \cite[prop.\ 2]{L}, and the first few cases are ubiquitous: $\lambda_3 = 1$, $\lambda_4 = \sqrt{2}$, $\lambda_5 = \frac{ \sqrt{5}+1 }{ 2 }$, $\lambda_6 = \sqrt{3}$.

It is a classic result by Conway and Coxeter that there is a bijection between triangulations of the $( n+3 )$-gon and integral frieze patterns of width $n$, see  \cite{CC1} and \cite{CC2}, items (28) and (29).  This is the special case $p=3$ of the following.

{\bf Theorem A. }
{\em 
There is a bijection between $p$-angulations of the $( n+3 )$-gon and frieze patterns of type $\Lambda_p$ and width $n$.
}

The proof uses less elementary means than \cite{CC1} and \cite{CC2}, relying on the theory of Hecke groups and the ensuing tilings of the hyperbolic plane by ideal $p$-angles, see Sections~\ref{sec:Farey} and~\ref{sec:ThmA}.  The skeleta of these tilings are the so-called Farey graphs $\cF_p$, and our approach generalises \cite[sec.\ 2]{MOT} which used $\cF_3$ to explain the result by Conway and Coxeter.

A {\em polygon dissection} $D$ of a polygon $P$ is a set of pairwise non-crossing diagonals.  It splits $P$ into subpolygons $P_1, \ldots, P_s$, see Figure~\ref{fig:four-angulation} on page~\pageref{fig:four-angulation}.  Observe that $D$ is a $p$-angulation if and only if each $P_i$ is a $p$-angle.  The following yields more general frieze patterns than Theorem A.

{\bf Theorem B. }
{\em 
There is an injection from polygon dissections of the $( n+3 )$-gon $P$ to frieze patterns of width $n$.  A dissection $D$ into subpolygons $P_i$, where $P_i$ is a $p_i$-gon, is mapped to a frieze pattern $F$ consisting of elements of $\cO_K$, the ring of algebraic integers of the field $K = \BQ( \lambda_{ p_1 }, \ldots, \lambda_{ p_s } )$.
}

The bijection in Theorem A is obtained by restricting the injection in Theorem B, which we now describe. 

\begin{Construction}
\label{con:Phi1}
A dissection $D$ of $P$ into subpolygons $P_i$, where $P_i$ is a $p_i$-gon, is mapped to a frieze pattern $F$ constructed as follows:  To each vertex $\alpha$ of $P$, associate the sum
\begin{equation}
\label{equ:the_recipe}
  \sum_{\mbox{\small $P_i$ is incident with $\alpha$}}\lambda_{ p_i }.
\end{equation}
When $\alpha$ cycles through the vertices of $P$ in the positive direction, this gives an $( n+3 )$-periodic sequence of elements of $\cO_K$.  This is the quiddity row of $F$.  See Remark~\ref{rmk:Phi2} for further details.
\hfill $\Box$
\end{Construction}

For example, consider the $4$-angulation in Figure~\ref{fig:four-angulation}.  Each subpolygon $P_i$ is a $4$-angle and $\lambda_4 = \sqrt{2}$.  Starting at vertex $0$, the elements associated to the vertices by Equation~\eqref{equ:the_recipe} are
\[
  \sqrt{2}, 2\sqrt{2}, \sqrt{2}, \sqrt{2}, 3\sqrt{2}, 2\sqrt{2}, \sqrt{2}, \sqrt{2}, 2\sqrt{2}, 2\sqrt{2}.
\]
Repeating this periodically gives the quiddity row of the frieze pattern in Figure~\ref{fig:sqrt2_frieze}.  Note that the occurrence of square roots in a regular pattern is an artefact of the $4$-angulation.  For example, the frieze pattern arising from the dissection of a $7$-gon into a triangle and a hexagon does not possess such regularity, see Figure~\ref{fig:sqrt3_frieze}.

The paper is organised as follows: Section~\ref{sec:friezes_on_polygons} recalls how frieze patterns correspond to friezes on polygons; this is essential to the rest of the paper.  Section~\ref{sec:Euclid} introduces a class of building blocks in the form of friezes defined by Euclidean lengths of diagonals.  They are glued to bigger friezes in Section~\ref{sec:gluing}.  Section~\ref{sec:ThmB} uses gluing to prove Theorem B.  Section~\ref{sec:Farey} recalls the Farey graph and proves a lemma on paths in such graphs. This is used to prove Theorem A in Section~\ref{sec:ThmA}.  Section~\ref{sec:questions} poses some questions.

\begin{figure}
\begingroup
\[
  \vcenter{
  \renewcommand{\objectstyle}{\scriptstyle}
  \xymatrix @-0.2pc @!0 {
    && 0 && 0 && 0 && 0 && 0 && 0 && 0 && 0 && 0 && 0 & \\
    &&& 1 && 1 && 1 && 1 && 1 && 1 && 1 && 1 && 1 && 1 \\
    && 2\sqrt{2} && \sqrt{2} && \sqrt{2} && 2\sqrt{2} && 2\sqrt{2} && \sqrt{2} && 2\sqrt{2} && \sqrt{2} && \sqrt{2} && 3\sqrt{2} & \\        
    &&& 3 && 1 && 3 && 7 && 3 && 3 && 3 && 1 && 5 && 11 & \\
    && 8\sqrt{2} && \sqrt{2} && \sqrt{2} && 5\sqrt{2} && 5\sqrt{2} && 4\sqrt{2} && 2\sqrt{2} && \sqrt{2} && 2\sqrt{2} && 9\sqrt{2} & \\    
    \cdots &&& 5 && 1 && 3 && 7 && 13 && 5 && 1 && 3 && 7 && 13 && \cdots \\
    && 4\sqrt{2} && 2\sqrt{2} && \sqrt{2} && 2\sqrt{2} && 9\sqrt{2} && 8\sqrt{2} && \sqrt{2} && \sqrt{2} && 5\sqrt{2} && 5\sqrt{2} & \\    
    &&& 3 && 3 && 1 && 5 && 11 && 3 && 1 && 3 && 7 && 3 & \\
    && \sqrt{2} && 2\sqrt{2} && \sqrt{2} && \sqrt{2} && 3\sqrt{2} && 2\sqrt{2} && \sqrt{2} && \sqrt{2} && 2\sqrt{2} && 2\sqrt{2} & \\    
    &&& 1 && 1 && 1 && 1 && 1 && 1 && 1 && 1 && 1 && 1 \\    
    && 0 && 0 && 0 && 0 && 0 && 0 && 0 && 0 && 0 && 0 & \\    
                        }
          }
\]
\endgroup
\caption{A frieze pattern of type $\Lambda_4$ and width $7$ with quiddity row $\sqrt{2}$, $2\sqrt{2}$, $\sqrt{2}$, $\sqrt{2}$, $3\sqrt{2}$, $2\sqrt{2}$, $\sqrt{2}$, $\sqrt{2}$, $2\sqrt{2}$, $2\sqrt{2}$ (repeating).  The pattern arises from the $4$-angulation in Figure~\ref{fig:four-angulation}, see Construction~\ref{con:Phi1}.  The width counts the number of non-trivial rows in the middle, not the rows of zeroes and ones at the bottom and top.}
\label{fig:sqrt2_frieze}
\end{figure}

\begin{figure}
\begingroup
\[
  \begin{tikzpicture}[auto]
    \node[name=s, shape=regular polygon, regular polygon sides=10, shape border rotate = 360/20, minimum size=6cm, draw] {}; 
    \draw[shift=(s.corner 1)] node[above] {$0$};
    \draw[shift=(s.corner 10)] node[above] {$1$};
    \draw[shift=(s.corner 9)] node[right] {$2$};
    \draw[shift=(s.corner 8)] node[right] {$3$};
    \draw[shift=(s.corner 7)] node[below] {$4$};
    \draw[shift=(s.corner 6)] node[below] {$5$};
    \draw[shift=(s.corner 5)] node[below] {$6$};
    \draw[shift=(s.corner 4)] node[left] {$7$};
    \draw[shift=(s.corner 3)] node[left] {$8$};
    \draw[shift=(s.corner 2)] node[above] {$9$};

    \draw[thick] (s.corner 7) to (s.corner 10);
    \draw[thick] (s.corner 2) to (s.corner 7);
    \draw[thick] (s.corner 3) to (s.corner 6);    
    
    \draw (0.65,0.9) node {$P_1$};
    \draw (2.35,0.1) node {$P_2$};
    \draw (-0.85,-0.25) node {$P_3$};
    \draw (-1.7,-1.5) node {$P_4$};    
  \end{tikzpicture} 
\]
\endgroup
\caption{A $4$-angulation splitting the $10$-gon $P$ (vertices $0,\ldots,9$) into $4$-gons $P_1$ (vertices $0,1,4,9$), $P_2$ (vertices $1,2,3,4$), $P_3$ (vertices $4,5,8,9$), and $P_4$ (vertices $5,6,7,8$).}
\label{fig:four-angulation}
\end{figure}
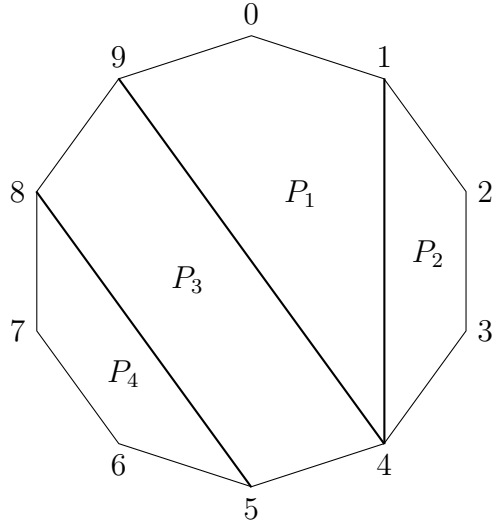

\begin{figure}
\begingroup
\[
  \vcenter{
  \renewcommand{\objectstyle}{\scriptstyle}
  \xymatrix @-0.1pc @!0 {
    && 0 && 0 && 0 && 0 && 0 && 0 && 0 && 0 && 0 && 0 & \\
    &&& 1 && 1 && 1 && 1 && 1 && 1 && 1 && 1 && 1 && 1 \\
    && 1 && 1+\sqrt{3} && \sqrt{3} && \sqrt{3} && \sqrt{3} && \sqrt{3} && 1+\sqrt{3} && 1 && 1+\sqrt{3} && \sqrt{3} & \\    
    &&& \sqrt{3} && 2+\sqrt{3} && 2 && 2 && 2 && 2+\sqrt{3} && \sqrt{3} && \sqrt{3} && 2+\sqrt{3} && 2 & \\
    {}\save[]+<0.7cm,0.3cm>*\txt<0pc>{$\scriptstyle\cdots$} \restore && 2 && 2 && 2+\sqrt{3} && \sqrt{3} && \sqrt{3} && 2+\sqrt{3} && 2 && 2 && 2 && 2+\sqrt{3} && {}\save[]+<-0.1cm,0.3cm>*\txt<0pc>{$\scriptstyle\cdots$} \restore \\
    &&& \sqrt{3} && \sqrt{3} && 1+\sqrt{3} && 1 && 1+\sqrt{3} && \sqrt{3} && \sqrt{3} && \sqrt{3} && \sqrt{3} && 1+\sqrt{3} & \\
    && 1 && 1 && 1 && 1 && 1 && 1 && 1 && 1 && 1 && 1 \\    
    &&& 0 && 0 && 0 && 0 && 0 && 0 && 0 && 0 && 0 && 0 & \\    
                        }
          }
\]
\endgroup
\caption{A frieze pattern of width $4$ with quiddity row $\sqrt{3}$, $\sqrt{3}$, $1+\sqrt{3}$, $1$, $1+\sqrt{3}$, $\sqrt{3}$, $\sqrt{3}$ (repeating).  The pattern arises from the dissection of a $7$-gon into a triangle and a hexagon, see Construction~\ref{con:Phi1}.}
\label{fig:sqrt3_frieze}
\end{figure}

\section{Friezes on polygons versus frieze patterns}
\label{sec:friezes_on_polygons}

This section shows that we can replace ``frieze pattern'' with ``frieze (on polygon)'' in Theorems A and B.  We need the following terminology.

\begin{Definition}
A {\em polygon} $P$ is a finite set $V$ of three or more {\em vertices} plus a cyclic ordering of $V$.

The predecessor and successor of $\alpha \in V$ are denoted $\alpha^-$ and $\alpha^+$.

An {\em edge} of $P$ is a subset $\{ \alpha,\alpha^+ \} \subset V$.  A {\em diagonal} of $P$ is a subset $\{ \alpha,\beta \} \subset V$ with $\beta \not\in \{ \alpha^-,\alpha,\alpha^+ \}$.

The diagonals $\{ \alpha,\beta \}$ and $\{ \gamma,\delta \}$ are said to {\em cross} if $\alpha,\beta,\gamma,\delta$ are four distinct vertices appearing in $V$ in either of the orders $\alpha,\gamma,\beta,\delta$ and $\alpha,\delta,\beta,\gamma$.

If $V$ has $p$ elements, then $P$ is called a {\em $p$-gon}.  It can be realised geometrically as a $p$-angle in the Euclidean plane, or as an ideal $p$-angle in the hyperbolic plane, see Remark~\ref{rmk:Farey} and Figure~\ref{fig:Farey}.
\hfill $\Box$
\end{Definition}

\begin{Definition}
\label{def:frieze}
Let $P$ be a polygon with vertex set $V$.  A {\em frieze} on $P$ is a map $f: V \times V \rightarrow [0,\infty[$ with the following properties.
\begin{enumerate}
\setlength\itemsep{4pt}

  \item  $f( \alpha,\beta ) = 0 \Leftrightarrow \alpha = \beta$.
  
  \item  $f( \alpha,\alpha^+ ) = 1$.

  \item  $f( \alpha,\beta ) = f( \beta,\alpha )$.
  
  \item  If $\{ \alpha,\beta \}$ and $\{ \gamma,\delta \}$ are crossing diagonals of $P$, then we have the Ptolemy relation
\begin{equation*}
\tag*{$\Box$}
  f( \alpha,\beta )f( \gamma,\delta ) = f( \alpha,\gamma )f( \beta,\delta ) + f( \alpha,\delta )f( \gamma,\beta ).
\end{equation*}

\end{enumerate}
\end{Definition}

\begin{Definition}
Rotating a frieze pattern $F$ by $45$ degrees places it as a diagonal
strip in the plane.  We will equip it with a coordinate system in
matrix style as shown in Figure~\ref{fig:coordinates}, with the first
coordinate increasing down, the second across, and $F( \alpha,\beta )$
will denote the entry of $F$ at position $( \alpha,\beta )$.
\begin{figure}
\begingroup
\begin{center}
  \begin{tikzpicture}[scale=1.1]

    \node (a) at (-1,5) { \tiny $(-n-3,-1)$ };
    \node (aa) at (0,4) { \tiny $(-n-2,0)$ };    
    \node (b) at (-1,1) { \tiny $(-1,-1)$ };
    \node (c) at (3,1) { \tiny $(-1,n+1)$ };
    \node (cc) at (2,2) { \tiny $(-2,n)$ };
    \node (d) at (0,0) { \tiny $(0,0)$ };
    \node (dd) at (1,-1) { \tiny $(1,1)$ };
    \node (e) at (4,0) { \tiny $(0,n+2)$ };
    \node (f) at (5,-1) { \tiny $(1,n+3)$ };    
    \node (ff) at (6,-2) { \tiny $(2,n+4)$ };    
    \node (g) at (4,-4) { \tiny $(n+2,n+2)$ };
    \node (gg) at (3,-3) { \tiny $(n+1,n+1)$ };    
    \node (h) at (5,-5) { \tiny $(n+3,n+3)$ };
    \node (i) at (9,-5) { \tiny $(n+3,2n+5)$ };
    \node (ii) at (8,-4) { \tiny $(n+2,2n+4)$ };

    \draw[dashed] (a) -- (b) -- (c) -- (cc) -- (aa) -- (a);
    \draw[dashed] (d) -- (e) -- (g) -- (gg) -- (dd) -- (d);
    \draw[dashed] (f) -- (h) -- (i) -- (ii) -- (ff) -- (f);
    
    \draw[thick,dotted] (-1.9,5.9) -- (-1.6,5.6);
    \draw[thick,dotted] (-1.9,1.9) -- (-1.6,1.6);
    \draw[thick,dotted] (5.6,-5.6) -- (5.9,-5.9);    
    \draw[thick,dotted] (9.6,-5.6) -- (9.9,-5.9);    
            
  \end{tikzpicture}
\end{center}
\endgroup
\caption{Rotating a frieze pattern $F$ of width $n$ by $45$ degrees places it as a diagonal strip in the plane.  We equip it with a coordinate system in matrix style as shown.  Coxeter proved that $F$ can be tiled by glide reflections of a triangular fundamental domain as indicated.  One of the copies of the domain has coordinates $( \alpha,\beta )$ with $0 \leqslant \alpha \leqslant \beta \leqslant n+2$.}
\label{fig:coordinates}
\end{figure}
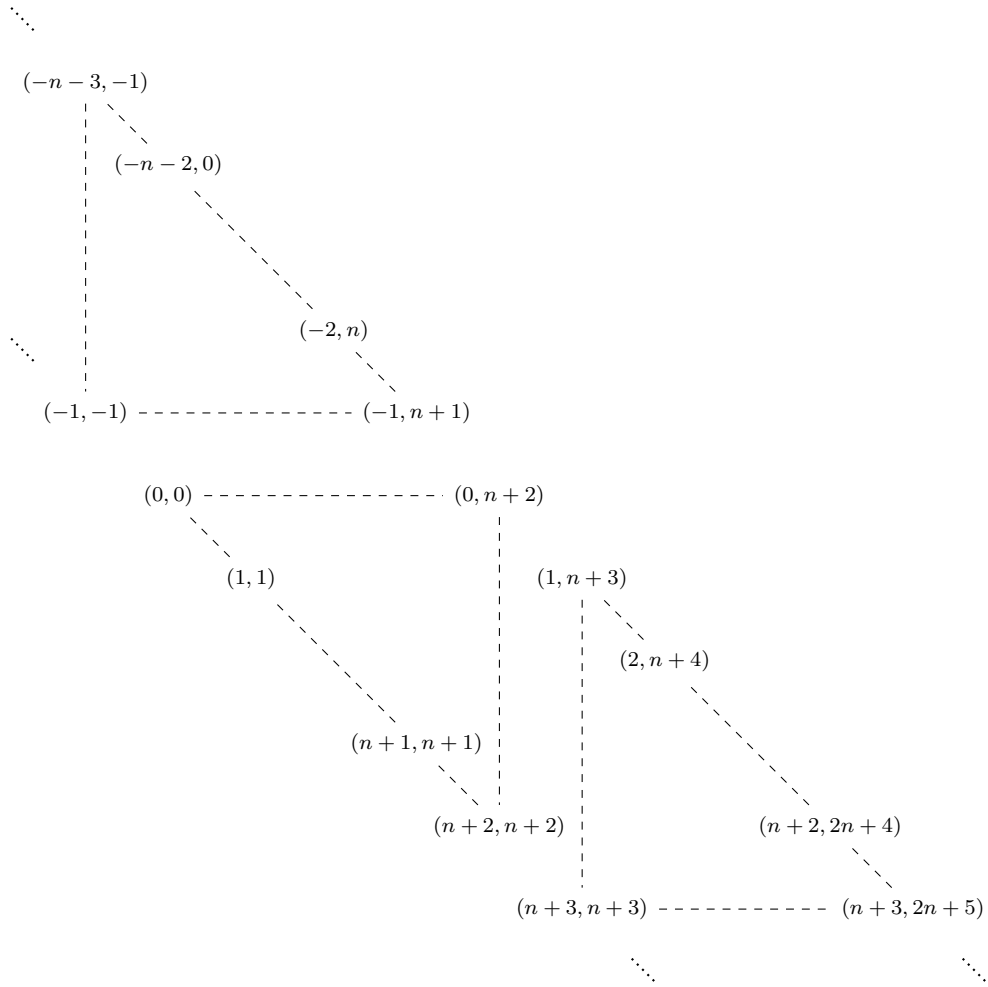
\hfill $\Box$
\end{Definition}

\begin{Remark}
\label{rmk:fundamental_domain}
Coxeter proved in \cite[sec.\ 6]{C} that a frieze pattern $F$ of width $n$ can be tiled by glide reflections of a triangular fundamental domain as shown in Figure~\ref{fig:coordinates}.  One of the copies of the domain has coordinates $( \alpha,\beta )$ with $0 \leqslant \alpha \leqslant \beta \leqslant n+2$.
\hfill $\Box$
\end{Remark}

We now describe a well-known bijection between frieze patterns and friezes.  We do not know a reference, but provide a sketch of the proof.

\begin{Theorem}
\label{thm:frieze_patterns_and_friezes}
Let $P$ be an $( n+3 )$-gon with vertices $V = \{ 0,1, \ldots, n+2 \}$
where the numbering reflects the cyclic ordering of $V$.  There is a
bijection $F \mapsto \Pi( F )$ from frieze patterns of width $n$ to
friezes on $P$, where $f = \Pi( F )$ is the map $f : V \times V \rightarrow [0,\infty[$ defined by
\[
  f( \alpha,\beta ) = 
  \left\{
    \begin{array}{lc}
      F( \alpha,\beta ) & \mbox{for $\alpha \leqslant \beta$,} \\[2mm]
      F( \beta,\alpha ) & \mbox{for $\alpha > \beta$.}
    \end{array}
  \right.
\]
The inequalities refer to the canonical linear ordering of $V$.
\end{Theorem}

\begin{proof}
Let $\cP$ be the set of maps $f : V \times V \rightarrow [0,\infty[$ which can be written $f = \Pi( F )$ for a frieze pattern $F$.  Since $\Pi$ is clearly injective, it remains to show that $\cP$ consists precisely of the friezes on $P$.

On the one hand, let $f$ be a frieze on $P$.  We must show $f \in \cP$, that is, $f = \Pi( F )$ for a frieze pattern $F$.  Define a candidate fundamental domain by setting $F( \alpha,\beta ) = f( \alpha,\beta )$ for $0 \leqslant \alpha \leqslant \beta \leqslant n+2$, where the inequalities refer to the canonical linear ordering of $V$, cf.\ Remark~\ref{rmk:fundamental_domain}.  Extend to a candidate frieze pattern $F$ by using the tiling by glide reflections shown in Figure~\ref{fig:coordinates}.  Definition~\ref{def:frieze} implies that $F$ is indeed a frieze pattern, and it is clear that $f = \Pi( F )$.

On the other hand, let $f \in \cP$ be given, that is, $f = \Pi( F )$
for a frieze pattern $F$.  We must show that $f$ is a frieze.
Definition~\ref{def:frieze}(i)-(iii) is immediate.  For part (iv) of
the definition, if $\{ \alpha,\beta \}$ and $\{ \gamma,\delta \}$ are crossing diagonals of $P$, then we can assume $0 \leqslant \alpha < \gamma < \beta < \delta \leqslant n+2$, where the inequalities refer to the canonical linear ordering of $V$.  Rotating $F$ by $45$ degrees as shown in Figure~\ref{fig:coordinates}, we can view $F$ as a partially defined bi-infinite matrix.  Definition~\ref{def:frieze_pattern}(i) means that each adjacent $2 \times 2$-submatrix has determinant $1$, so $F$ is a so-called partial $\SL2$-tiling, see \cite[p.\ 3139]{ARS} and \cite[def.\ 3.1]{HJ}.  Moreover, the proof of \cite[prop.\ 1]{BR} shows that $F$ is tame, that is, each adjacent $3 \times 3$-submatrix has determinant $0$.  Hence the methods of \cite[sec.\ 5]{HJ} apply to $F$, and in particular, the proof of \cite[prop.\ 5.7]{HJ} shows
\[
  \begin{vmatrix}
    F( \alpha,\beta ) & F( \alpha,\delta ) \\[2.5mm]
    F( \gamma,\beta ) & F( \gamma,\delta )
  \end{vmatrix}
  =
  \begin{vmatrix}
    F( \alpha,\gamma ) & F( \alpha,\gamma+1 ) \\[2.5mm]
    F( \gamma,\gamma ) & F( \gamma,\gamma+1 )
  \end{vmatrix}
  \begin{vmatrix}
    F( \beta,\delta ) & F( \delta,\delta ) \\[2.5mm]
    F( \beta,\delta+1 ) & F( \delta,\delta+1 )
  \end{vmatrix}.
\]
But $F( \gamma,\gamma ) = F( \delta,\delta ) = 0$ and $F( \gamma,\gamma+1 ) = F( \delta,\delta+1 ) = 1$, so this gives 
\[
  F( \alpha,\beta )F( \gamma,\delta ) 
  - F( \alpha,\delta )F( \gamma,\beta )
  = F( \alpha,\gamma )F( \beta,\delta ).
\]
Here $F$ can be replaced by $f$ because $f = \Pi( F )$, showing that Definition~\ref{def:frieze}(iv) holds.
\end{proof}

Theorem~\ref{thm:frieze_patterns_and_friezes} means that we can replace ``frieze pattern'' with ``frieze'' in Theorem B.

\begin{Definition}
\label{def:type_Lambda_q_frieze}
A frieze $f$ is of {\em type $\Lambda_p$} if each $f( \alpha^-,\alpha^+ )$ is a (necessarily positive) integral multiple of $\lambda_p$.
\hfill $\Box$
\end{Definition}

Theorem~\ref{thm:frieze_patterns_and_friezes} and Definitions~\ref{def:type_Lambda_q_frieze_pattern} and~\ref{def:type_Lambda_q_frieze} imply:

\begin{Lemma}
\label{lem:type_Lambda_q}
Let $F$ be a frieze pattern, $f = \Pi( F )$ the corresponding frieze under the bijection in Theorem~\ref{thm:frieze_patterns_and_friezes}.  Then $F$ and $f$ are of type $\Lambda_p$ simultaneously.
\end{Lemma}

This means that we can also replace ``frieze pattern'' with ``frieze'' in Theorem A.

The following lemma is an immediate consequence of Theorem~\ref{thm:frieze_patterns_and_friezes} and the coordinate system in Figure~\ref{fig:coordinates}.

\begin{Lemma}
\label{lem:quiddity2}
Let $F$ be a frieze pattern, $f = \Pi( F )$ the corresponding frieze under the bijection in Theorem~\ref{thm:frieze_patterns_and_friezes}.  The quiddity row of $F$ is formed by the numbers $f( \alpha^-,\alpha^+ )$ when $\alpha$ cycles through the vertices of the polygon $P$ in the positive direction.
\end{Lemma}

\begin{Lemma}
\label{lem:agree}
If $f$ and $f'$ are friezes on a polygon $P$ such that $f( \alpha^-,\alpha^+ ) = f'( \alpha^-,\alpha^+ )$ for each vertex $\alpha$, then $f = f'$.
\end{Lemma}

\begin{proof}
Theorem~\ref{thm:frieze_patterns_and_friezes} says that $f$ and $f'$ correspond to frieze patterns $F$ and $F'$, which have the same quiddity row by Lemma~\ref{lem:quiddity2}.  Hence $F = F'$ so $f = f'$.  
\end{proof}

\section{Friezes on polygons induced by Euclidean lengths of diagonals}
\label{sec:Euclid}

The following frieze appeared already in \cite[(16)]{CC2} and
\cite[sec.\ 5.3]{C2}.

\begin{Definition}
\label{def:ell_p}
Let $p \geqslant 3$ be an integer, $P$ a $p$-gon.  Let $P$ be realised as a regular $p$-angle with edges of length $1$ in the Euclidean plane.  If $\alpha,\beta$ are vertices of $P$, then set
\[
\tag*{$\Box$}
  \ell_p( \alpha,\beta ) = \mbox{the length of the line segment from $\alpha$ to $\beta$}.
\]
\end{Definition}

\begin{Lemma}
\label{lem:ell_q}
\begin{enumerate}
\setlength\itemsep{4pt}

  \item  $\ell_p$ is a frieze of type $\Lambda_p$ on $P$.
  
  \item  $\ell_p$ has values in $\cO_K$, the ring of algebraic integers of the field $K = \BQ( \lambda_p )$.  
  
  \item  If $\alpha$ is a vertex of $P$, then $\ell_p( \alpha^-,\alpha^+ ) = \lambda_p$.
    
  \item  If $\alpha',\alpha''$ are non-consecutive vertices of $P$, then $\ell_p( \alpha',\alpha'' ) > 1$.

\end{enumerate}
\end{Lemma}

\begin{proof}
Part (iii) can be proved by elementary trigonometry, see also \cite[pp.\ 23--24]{S}.  Part (iv) is clear.

(i): It is clear that $\ell_p$ has values in the non-negative real
numbers and satisfies 
Definition~\ref{def:frieze}(i)-(iii), while condition (iv) in the definition is
Ptolemy's theorem, see \cite[thm.\ 2.61]{CG}.  Part (iii) of the lemma
shows that $\ell_p$ has type $\Lambda_p$.

(ii):  Theorem~\ref{thm:frieze_patterns_and_friezes} and Lemma~\ref{lem:type_Lambda_q} say that $\ell_p$ corresponds to a frieze pattern $F$ of type $\Lambda_p$.  It is enough to see that $F$ has entries in $\cO_K$, and it does because its quiddity row has entries in $\BN\lambda_p \subset \cO_K$, while \cite[eq.\ (6.6)]{C} expresses each entry of $F$ as the determinant of a matrix with entries which are either sampled from the quiddity row of $F$, or equal to $0$ or $1$.
\end{proof}

\section{Gluing friezes on polygons}
\label{sec:gluing}

\begin{Setup}
In this section, we consider a polygon $P$ with vertex set $V$ and a polygon dissection $D$ into subpolygons $P_1, \ldots, P_s$ with vertex sets $V_1, \ldots, V_s \subseteq V$ (polygon dissections were defined before Theorem B in the introduction).  Let $f_i$ be a frieze on $P_i$ for each $i$.
\hfill $\Box$
\end{Setup}

\begin{Lemma}
[Gluing friezes]
\label{lem:extension}
There is a unique frieze $f$ on $P$ that restricts to $f_i$ on $P_i$ for each $i$.  If each $f_i$ has values in $\cO_K$, the ring of algebraic integers of a field $K \subseteq \BR$, then so does $f$.
\end{Lemma}

\begin{proof}
By induction it is enough to prove the proposition for $s = 2$.  Then the dissection is given by a single diagonal $\{ \zeta,\eta \}$ and $V_1 \cap V_2 = \{ \zeta,\eta \}$.  Set $U_i = V_i \setminus \{ \zeta,\eta \}$.  The frieze $f$ must satisfy Definition~\ref{def:frieze}(iv) and restrict to $f_i$ on $V_i \times V_i$.  Hence the only possibility is to define $f( \alpha,\beta )$ by the following entries, according to whether $\alpha$ and $\beta$ are in $U_1$, $U_2$, or $\{ \zeta,\eta \}$.
\[
\mbox{
\begin{tabular}{c|cccccc}
  \diaghead{\theadfont xxxxxxxxx}{$\alpha$}{$\beta$} & $U_1$ & $\{ \zeta,\eta \}$ & $U_2$ \\[2mm] \cline{1-4}
  $U_1$ & $f_1( \alpha,\beta )$ & $f_1( \alpha,\beta )$ & $f_1( \zeta,\alpha )f_2( \eta,\beta ) + f_1( \eta,\alpha )f_2( \zeta,\beta )$ \\[2mm]
  $\{ \zeta,\eta \}$ & $f_1( \alpha,\beta )$ & $f_1( \alpha,\beta ) = f_2( \alpha,\beta )$ & $f_2( \alpha,\beta )$ \\[2mm]
  $U_2$ & $f_1( \zeta,\beta )f_2( \eta,\alpha ) + f_1( \eta,\beta )f_2( \zeta,\alpha )$ & $f_2( \alpha,\beta )$ & $f_2( \alpha,\beta )$ \\
\end{tabular}
}
\]
If each $f_i$ has values in $\cO_K$, then clearly so does $f$.  We leave as an exercise the long winded, but elementary verification that $f$ is indeed a frieze.
\end{proof}

\begin{Lemma}
\label{lem:ones}
For each $i$, suppose that if $\alpha',\alpha''$ are non-consecutive vertices of $P_i$, then $f_i( \alpha',\alpha'' ) > 1$.  Let $\alpha,\beta$ be non-consecutive vertices of $P$.  Then
\[
  f( \alpha,\beta ) = 1 \Leftrightarrow \mbox{the diagonal $\{ \alpha,\beta \}$ is in the dissection $D$}.
\]
\end{Lemma}

\begin{proof}
If $\{ \alpha,\beta \}$ is in $D$, then $\alpha$ and $\beta$ are consecutive vertices of one of the subpolygons $P_i$ whence $f( \alpha,\beta ) = f_i( \alpha,\beta ) = 1$ by Definition~\ref{def:frieze}(ii).

If $\{ \alpha,\beta \}$ is not in $D$, then it crosses $m \geqslant 0$ diagonals in $D$, and we prove $f( \alpha,\beta ) > 1$ by induction on $m$.

Suppose $m=0$.  Then $\{ \alpha,\beta \}$ is a diagonal of one of the subpolygons $P_i$, so $f( \alpha,\beta ) = f_i( \alpha,\beta ) = ( \ast )$.  Since $\alpha,\beta$ are non-consecutive vertices of $P$ and since $\{ \alpha,\beta \}$ is not in $D$, it follows that $\alpha,\beta$ are non-consecutive vertices of $P_i$ whence $( \ast ) > 1$ by the assumption in the lemma.

Suppose $m \geqslant 1$.  Let $\{ \gamma,\delta \}$ be a diagonal in $D$ crossed by $\{ \alpha,\beta \}$.  Definition~\ref{def:frieze}(iv) gives
\[
  f( \alpha,\beta )f( \gamma,\delta )
  = f( \alpha,\gamma )f( \beta,\delta )
  + f( \alpha,\delta )f( \gamma,\beta ),
\]
and by the first two lines of the proof, this reads
\begin{equation}
\label{equ:ones}
  f( \alpha,\beta )
  = f( \alpha,\gamma )f( \beta,\delta )
  + f( \alpha,\delta )f( \gamma,\beta ).
\end{equation}
Each of $\{ \alpha,\gamma \}$, $\{ \beta,\delta \}$, $\{ \alpha,\delta \}$, $\{ \gamma,\beta \}$ crosses fewer than $m$ diagonals in $D$, see Figure~\ref{fig:Ptolemy}.
\begin{figure}
\begingroup
\[
  \begin{tikzpicture}[auto]
    \node[name=s, shape=regular polygon, regular polygon sides=20, minimum size=6cm, draw] {}; 
    \draw (1.5*360/20:3.25cm) node {$\alpha$};
    \draw (7.5*360/20:3.25cm) node {$\delta$};
    \draw (11.5*360/20:3.25cm) node {$\beta$};
    \draw (18.5*360/20:3.25cm) node {$\gamma$};    
    \draw[thick] (7.5*360/20:3cm) to (18.5*360/20:3cm);
    \draw[thick] (1.5*360/20:3cm) to (11.5*360/20:3cm);
    \draw[thick] (1.5*360/20:3cm) to (7.5*360/20:3cm);
    \draw[thick] (1.5*360/20:3cm) to (18.5*360/20:3cm);
    \draw[thick] (11.5*360/20:3cm) to (7.5*360/20:3cm);
    \draw[thick] (11.5*360/20:3cm) to (18.5*360/20:3cm);    
  \end{tikzpicture} 
\]
\endgroup
\caption{The diagonal $\{ \alpha,\beta \}$ is assumed to cross $m \geqslant 1$ diagonals in the polygon dissection $D$, among them $\{ \gamma,\delta \}$.  Each of $\{ \alpha,\gamma \}$, $\{ \beta,\delta \}$, $\{ \alpha,\delta \}$, $\{ \gamma,\beta \}$ only crosses diagonals in $D$ also crossed by $\{ \alpha,\beta \}$.  Hence each of $\{ \alpha,\gamma \}$, $\{ \beta,\delta \}$, $\{ \alpha,\delta \}$, $\{ \gamma,\beta \}$ crosses fewer than $m$ diagonals in $D$.}
\label{fig:Ptolemy}
\end{figure}
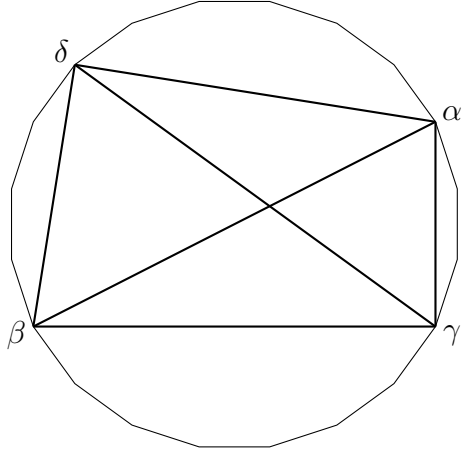
If $\alpha$ and $\gamma$ are non-consecutive vertices of $P$, then by induction $f( \alpha,\gamma ) > 1$.  If they are consecutive, then $f( \alpha,\gamma ) = 1$.  In any event, $f( \alpha,\gamma ) \geqslant 1$, and similarly $f( \beta,\delta ), f( \alpha,\delta ), f( \gamma,\beta ) \geqslant 1$.  But then Equation~\eqref{equ:ones} gives the first of the following inequalities:
$f( \alpha,\beta ) \geqslant 2 > 1$, completing the induction.
\end{proof}

\begin{Lemma}
\label{lem:quiddity}
Suppose there are numbers $c_1, \ldots, c_s$ such that, for each $i$, if $\alpha',\alpha,\alpha''$ are consecutive vertices of $P_i$, then $f_i( \alpha',\alpha'' ) = c_i$.  Then for each $\alpha \in V$,
\[
  f( \alpha^-,\alpha^+ )
  = \sum_{\mbox{\small $P_i$ is incident with $\alpha$}}c_i.
\]
\end{Lemma}

\begin{proof}
If $P_{ i_1 }, \ldots, P_{ i_m }$ are those of the $P_i$ which are incident with $\alpha$, then we must show
\begin{equation}
\label{equ:quiddity}
  f( \alpha^-,\alpha^+ ) = c_{ i_1 } + \cdots + c_{ i_m }.
\end{equation}
We do so by induction on $m$.

Suppose $m=1$.  Then $\alpha$ is incident with $P_{ i_1 }$ only, whence $\alpha^-,\alpha,\alpha^+$ are consecutive vertices of $P_{ i_1 }$.  This implies $f( \alpha^-,\alpha^+ ) = f_{ i_1 }( \alpha^-,\alpha^+ ) = c_{ i_1 }$, proving Equation~\eqref{equ:quiddity}.

Suppose $m \geqslant 2$.  As illustrated by Figure~\ref{fig:quiddity}, there is a vertex $\beta$ with the following properties: $\{ \alpha,\beta \}$ is a diagonal in $D$, which divides $P$ into subpolygons $P'$ with vertices $\{ \alpha,\alpha^+, \ldots, \beta^-,\beta \}$ and $P''$ with vertices $\{ \beta,\beta^+, \ldots, \alpha^-,\alpha \}$, such that $P_{ i_1 }$, $\ldots$, $P_{ i_{ m-1 } }$ are subpolygons of $P'$ while $P_{ i_m }$ is a subpolygon of $P''$.
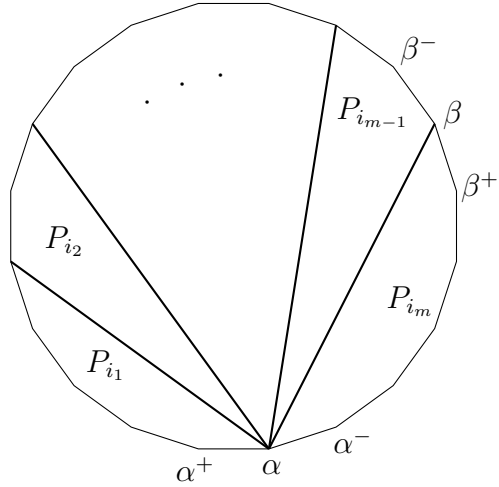
\begin{figure}
\begingroup
\[
  \begin{tikzpicture}[auto]
    \node[name=s, shape=regular polygon, regular polygon sides=20, minimum size=6cm, draw] {}; 
    \draw (14.5*360/20:3.27cm) node {$\alpha^+$};
    \draw (15.5*360/20:3.25cm) node {$\alpha$};
    \draw (16.5*360/20+2:3.27cm) node {$\alpha^-$};
    \draw (0.5*360/20:3.32cm) node {$\beta^+$};    
    \draw (1.5*360/20:3.27cm) node {$\beta$};    
    \draw (2.5*360/20-1.5:3.40cm) node {$\beta^-$};    
    \draw[thick] (15.5*360/20:3cm) to (1.5*360/20:3cm);
    \draw[thick] (15.5*360/20:3cm) to (3.5*360/20:3cm);
    \draw[thick] (15.5*360/20:3cm) to (8.5*360/20:3cm);
    \draw[thick] (15.5*360/20:3cm) to (10.5*360/20:3cm);
    \draw (-1.7,-1.85) node {$P_{ i_1 }$};
    \draw (-2.25,-0.2) node {$P_{ i_2 }$};
    \draw (1.85,1.5) node {$P_{ i_{ m-1 } }$};
    \draw (2.35,-1) node {$P_{ i_m }$};
    \draw (95:2cm) node {$\cdot$};
    \draw (110:2cm) node {$\cdot$};
    \draw (125:2cm) node {$\cdot$};    
  \end{tikzpicture} 
\]
\endgroup
\caption{Computing the value $f( \alpha^-,\alpha^+ )$ when the frieze $f$ on $P$ restricts to friezes $f_i$ on the $P_i$.}
\label{fig:quiddity}
\end{figure}
By Lemma~\ref{lem:extension} there is a unique frieze $f'$ on $P'$ which restricts to $f_i$ on each $P_i$ which is a subpolygon of $P'$.  Viewed in $P'$, the vertices $\beta,\alpha,\alpha^+$ are consecutive while $\alpha$ is incident with $P_{ i_1 }, \ldots, P_{ i_{ m-1 } }$.  Hence by induction,
\[
  f'( \beta,\alpha^+ ) = c_{ i_1 } + \cdots + c_{ i_{ m-1 } }.
\]
Since $f$ is the unique frieze on $P$ which restricts to $f_i$ on $P_i$ for each $i$, it follows that $f$ restricts to $f'$ on $P'$.  Hence also
\begin{equation}
\label{equ:quiddity1}
  f( \beta,\alpha^+ ) = c_{ i_1 } + \cdots + c_{ i_{ m-1 } }.
\end{equation}

Since $\alpha^-,\alpha,\beta$ are consecutive vertices of $P_{ i_m }$ we have
\begin{equation}
\label{equ:quiddity3}
  f( \alpha^-,\beta ) = f_{ i_m }( \alpha^-,\beta ) = c_{ i_m }.
\end{equation}

Definition~\ref{def:frieze}(iv) implies
\[
  f( \alpha,\beta )f( \alpha^-,\alpha^+ )
  = f( \alpha,\alpha^- )f( \beta,\alpha^+ )
  + f( \alpha,\alpha^+ )f( \alpha^-,\beta ),
\]
and by Definition~\ref{def:frieze}(ii)-(iii), this reads
\begin{equation}
\label{equ:quiddity2}
  f( \alpha^-,\alpha^+ )
  = f( \beta,\alpha^+ ) + f( \alpha^-,\beta ).
\end{equation}

Combining Equations~\eqref{equ:quiddity1}, \eqref{equ:quiddity3}, and \eqref{equ:quiddity2} proves Equation~\eqref{equ:quiddity}.
\end{proof}

\section{Proof of Theorem B}
\label{sec:ThmB}

\begin{Definition}
\label{def:Phi}
Let $P$ be a polygon.  If a polygon dissection $D$ divides $P$ into subpolygons $P_1, \ldots, P_s$ where $P_i$ is a $p_i$-gon, then let $f_i$ denote $\ell_{ p_i }$ from Definition~\ref{def:ell_p} viewed as a frieze on $P_i$.  Lemma~\ref{lem:extension} gives a frieze $f$ on $P$ which restricts to $f_i$ on $P_i$.  Set $\Phi( D ) = f$.  
\hfill $\Box$
\end{Definition}

By virtue of Theorem~\ref{thm:frieze_patterns_and_friezes}, the following implies Theorem B.

\begin{Theorem}
\label{thm:B_prime}
\begin{enumerate}
\setlength\itemsep{4pt}

  \item  $\Phi$ is an injection from polygon dissections of $P$ to friezes on $P$.

  \item  In the situation of Definition~\ref{def:Phi}, the frieze $\Phi( D )$ has values in $\cO_K$, the ring of algebraic integers of the field $K = \BQ( \lambda_{ p_1 }, \ldots, \lambda_{ p_s } )$.

\end{enumerate}
\end{Theorem}

\begin{proof}
(i)  Observe that if $\alpha',\alpha''$ are non-consecutive vertices of $P_i$, then $f_i( \alpha',\alpha'' ) = \ell_{ p_i }( \alpha',\alpha'' ) > 1$ by Lemma~\ref{lem:ell_q}(iv).  If $\alpha,\beta$ are non-consecutive vertices of $P$, then Lemma~\ref{lem:ones} says that $\{ \alpha,\beta \}$ is in $D$ if and only if $f( \alpha,\beta ) = 1$, so $D$ can be recovered from $f$.  

(ii)  By Lemma~\ref{lem:extension} the frieze $\Phi( D ) = f$ has values in $\cO_K$ because so does each $\ell_{ p_i }$ by Lemma~\ref{lem:ell_q}(ii).  
\end{proof}

\begin{Remark}
\label{rmk:Phi2}
The frieze pattern $F$ in Theorem B is obtained from the frieze $\Phi( D ) = f$ on the polygon $P$ via Theorem~\ref{thm:frieze_patterns_and_friezes}.  This implies the description of $F$ given in Construction~\ref{con:Phi1}:

By Lemma~\ref{lem:quiddity2}, the quiddity row of $F$ is formed by the numbers $f( \alpha^-,\alpha^+ )$ when $\alpha$ cycles through the vertices of $P$ in the positive direction.

For each $i$, if $\alpha',\alpha,\alpha''$ are consecutive vertices of $P_i$, then $f_i( \alpha',\alpha'' ) = \ell_{ p_i }( \alpha',\alpha'' ) = \lambda_{ p_i }$ by Lemma~\ref{lem:ell_q}(iii).  Hence Lemma~\ref{lem:quiddity} applies with $c_i = \lambda_{ p_i }$, so we have
\begin{equation}
\label{equ:quiddity_sum}
  f( \alpha^-,\alpha^+ )
  = \sum_{\mbox{\small $P_i$ is incident with $\alpha$}} \lambda_{p_i}.
\end{equation}
This recovers Construction~\ref{con:Phi1}.

Note that if $D$ is a $p$-angulation then $p_i = p$ for each $i$.  Hence Equation \eqref{equ:quiddity_sum} shows that $f$ is of type $\Lambda_p$. 
\hfill $\Box$
\end{Remark}

\section{Paths in Farey graphs}
\label{sec:Farey}

\begin{Remark}
\label{rmk:Farey}
This remark recalls some material from \cite[sec.\ 1]{SW} for the convenience of the reader.

Consider the set $\overline{ \BC } = \BC \cup \{ \infty \}$ and the subsets $\BH = \{ z \in \BC \mid \Image( z ) > 0 \}$ and $\overline{ \BH } = \{ z \in \BC \mid \Image( z ) \geqslant 0 \} \cup \{ \infty \}$.

The open upper half plane $\BH$ is a model for the {\em hyperbolic plane}, and $\overline{ \BH }$ is a model for the hyperbolic plane and its {\em ideal boundary} which is identified with $\BR \cup \{ \infty \}$.

The group $\Mob( \overline{ \BC } )$ of {\em M\"{o}bius transformations} consists of the maps $\overline{ \BC } \rightarrow \overline{ \BC }$ of the form $z \mapsto \frac{ az+b }{ cz+d }$ with $a,b,c,d \in \BC$.  There is a surjective group homomorphism
\begin{equation}
\label{equ:mu}
  \SL2( \BC )
  \stackrel{ \mu }{ \longrightarrow }
  \Mob( \overline{ \BC } )
  \;\; , \;\;
  \mu \begin{pmatrix} a & b \\ c & d \end{pmatrix}
  = \Big( z \mapsto \frac{ az+b }{ cz+d } \Big).
\end{equation}
The kernel is $\{ \pm I \}$ where $I$ is the identity matrix.  The subgroup $\Mob( \BH ) \subset \Mob( \overline{ \BC } )$ consists of the M\"{o}bius transformations mapping $\BH$ to itself, or equivalently, mapping $\overline{ \BH }$ to itself.  It consists of the maps $z \mapsto \frac{ az+b }{ cz+d }$ with $a,b,c,d \in \BR$, and $\mu$ restricts to a surjective group homomorphism 
\[
  \SL2( \BR ) \longrightarrow \Mob( \BH ),
\]
again with kernel $\{ \pm I \}$.

Let $p \geqslant 3$ be an integer.  The {\em Hecke group} $G_p$ is the discrete subgroup of $\Mob( \BH )$ generated by $\sigma$ and $\tau_p$ where
\[
  \sigma( z ) = -\frac{ 1 }{ z }
  \;\;,\;\;
  \tau_p( z ) = z + \lambda_p,
\]
still with $\lambda_p = 2\cos( \frac{ \pi }{ p } )$.

Consider
\[
  L = \{ iy \mid 0 \leqslant y < \infty \} \cup \{ \infty \}
\]
which is a hyperbolic line in $\BH$ along with the ideal points $0$ and $\infty$ which it determines.  The {\em Farey graph} $\cF_p$ has vertices given by the $G_p$-orbit of $\infty$, and edges by the $G_p$-orbit of $L$.  It is the skeleton of a tiling of $\overline{ \BH }$ by ideal $p$-angles, that is, $p$-angles whose vertices are ideal points, see Figure~\ref{fig:Farey}.
\hfill $\Box$
\end{Remark}

\begin{Setup}
\label{set:Moebius}
In the rest of this section, let $n \geqslant 0$ be an integer, $q_0, \ldots, q_{ n+2 }$ positive integers.  Define M\"{o}bius transformations by
\[
  \xi_{ \alpha }( z ) = q_{ \alpha }\lambda_p - \frac{ 1 }{ z }.
\]  
Observe that $\xi_{ \alpha } = \tau_p^{ q_{ \alpha } } \circ \sigma$ is in $G_p$, so
\[
  \upsilon_0 = \infty
  \;\; \mbox{and} \;\;
  \upsilon_{ \alpha } = \xi_0 \cdots \xi_{ \alpha-1 }( \infty )
\]  
are vertices of $\cF_p$ for $1 \leqslant \alpha \leqslant n+2$.
\hfill $\Box$
\end{Setup}

\begin{figure}[H]
\begingroup
\[
  \begin{tikzpicture}[scale=10]
    \draw (-0.05,0) to (1.464213562,0);
    \draw (-0.10,0) node {$\cdot$};
    \draw (-0.08,0) node {$\cdot$};
    \draw (-0.06,0) node {$\cdot$};
    \draw (1.514213562,0) node {$\cdot$};    
    \draw (1.494213562,0) node {$\cdot$};    
    \draw (1.474213562,0) node {$\cdot$};

    \draw[thick,blue] (0,0) to (0,0.8);
    \draw[thick,blue] (-0.01,0.49) to (0.0,0.5);
    \draw[thick,blue] (0.01,0.49) to (0.0,0.5);    
    \draw (0,0.85) node {$\cdot$};
    \draw (0,0.83) node {$\cdot$};
    \draw (0,0.81) node {$\cdot$};        
    \draw[thick,blue] (1.414213562,0) to (1.414213562,0.8);
    \draw[thick,blue] (1.414213562-0.01,0.5) to (1.414213562,0.49);
    \draw[thick,blue] (1.414213562+0.01,0.5) to (1.414213562,0.49);        
    \draw (1.414213562,0.85) node {$\cdot$};
    \draw (1.414213562,0.83) node {$\cdot$};
    \draw (1.414213562,0.81) node {$\cdot$};            
    
    \draw (1.414213562,0) arc (0:180:0.353553391cm);
    \draw[thick,blue] (1.414213562,0) arc (0:180:0.176776695cm);
    \draw[thick,blue] ($(1.414213562,0)+(-0.176776695cm,0.176776695cm)+(0.01,0.01)$) to ($(1.414213562,0)+(-0.176776695cm,0.176776695cm)$);
    \draw[thick,blue] ($(1.414213562,0)+(-0.176776695cm,0.176776695cm)+(0.01,-0.01)$) to ($(1.414213562,0)+(-0.176776695cm,0.176776695cm)$);    
    \draw (1.414213562,0) arc (0:180:0.11785113cm);
    \draw (1.178511302,0) arc (0:180:0.023570226cm);
    \draw (1.13137085,0) arc (0:180:0.035355339cm);
    \draw[thick,blue] (1.060660172,0) arc (0:180:0.058925565cm);
    \draw[thick,blue] ($(1.060660172,0)+(-0.058925565cm,0.058925565cm)+(0.01,0.01)$) to ($(1.060660172,0)+(-0.058925565cm,0.058925565cm)$);
    \draw[thick,blue] ($(1.060660172,0)+(-0.058925565cm,0.058925565cm)+(0.01,-0.01)$) to ($(1.060660172,0)+(-0.058925565cm,0.058925565cm)$);        
    \draw (1.060660172,0) arc (0:180:0.025253814cm);    
    \draw (1.010152545,0) arc (0:180:0.010101525cm);
    \draw (0.989949494,0) arc (0:180:0.023570226cm);
    \draw[thick,blue] (0.942809042,0) arc (0:180:0.11785113cm);
    \draw[thick,blue] ($(0.942809042,0)+(-0.11785113cm,0.11785113cm)+(0.01,0.01)$) to ($(0.942809042,0)+(-0.11785113cm,0.11785113cm)$);
    \draw[thick,blue] ($(0.942809042,0)+(-0.11785113cm,0.11785113cm)+(0.01,-0.01)$) to ($(0.942809042,0)+(-0.11785113cm,0.11785113cm)$);            
    \draw (0.942809042,0) arc (0:180:0.029462783cm);
    \draw (0.883883476,0) arc (0:180:0.01767767cm);
    \draw (0.848528137,0) arc (0:180:0.070710678cm);
    \draw (0.707106781,0) arc (0:180:0.353553391cm);
    \draw[thick,blue] (0.707106781,0) arc (0:180:0.11785113cm);
    \draw[thick,blue] ($(0.707106781,0)+(-0.11785113cm,0.11785113cm)+(0.01,0.01)$) to ($(0.707106781,0)+(-0.11785113cm,0.11785113cm)$);
    \draw[thick,blue] ($(0.707106781,0)+(-0.11785113cm,0.11785113cm)+(0.01,-0.01)$) to ($(0.707106781,0)+(-0.11785113cm,0.11785113cm)$);                
    \draw (0.707106781,0) arc (0:180:0.070710678cm);    
    \draw (0.565685425,0) arc (0:180:0.01767767cm);
    \draw (0.530330086,0) arc (0:180:0.029462783cm);
    \draw (0.471404521,0) arc (0:180:0.058925565cm);
    \draw[thick,blue] (0.471404521,0) arc (0:180:0.023570226cm);
    \draw[thick,blue] ($(0.471404521,0)+(-0.023570226cm,0.023570226cm)+(0.007,0.007)$) to ($(0.471404521,0)+(-0.023570226cm,0.023570226cm)$);
    \draw[thick,blue] ($(0.471404521,0)+(-0.023570226cm,0.023570226cm)+(0.007,-0.007)$) to ($(0.471404521,0)+(-0.023570226cm,0.023570226cm)$);                
    \draw[thick,blue] (0.424264069,0) arc (0:180:0.010101525cm);
    \draw[thick,blue] (0.404061018,0) arc (0:180:0.025253814cm);
    \draw[thick,blue] ($(0.404061018,0)+(-0.023570226cm,0.023570226cm)+(0.007,0.007)$) to ($(0.404061018,0)+(-0.023570226cm,0.023570226cm)$);
    \draw[thick,blue] ($(0.404061018,0)+(-0.023570226cm,0.023570226cm)+(0.007,-0.007)$) to ($(0.404061018,0)+(-0.023570226cm,0.023570226cm)$);                    
    \draw[thick,blue] (0.353553391,0) arc (0:180:0.176776695cm);
    \draw[thick,blue] ($(0.353553391,0)+(-0.176776695cm,0.176776695cm)+(0.01,0.01)$) to ($(0.353553391,0)+(-0.176776695cm,0.176776695cm)$);
    \draw[thick,blue] ($(0.353553391,0)+(-0.176776695cm,0.176776695cm)+(0.01,-0.01)$) to ($(0.353553391,0)+(-0.176776695cm,0.176776695cm)$);                    
    \draw (0.353553391,0) arc (0:180:0.035355339cm);
    \draw (0.282842712,0) arc (0:180:0.023570226cm);
    \draw (0.23570226,0) arc (0:180:0.11785113cm);
    
    \draw (1.414213562,-0.028) node {\tiny $\upsilon_1$};
    \draw (1.060660172,-0.028) node {\tiny $\upsilon_2$};
    \draw (0.942809042,-0.028) node {\tiny $\upsilon_3$};
    \draw (0.707106781,-0.028) node {\tiny $\upsilon_4$};
    \draw (0.471404521,-0.028) node {\tiny $\upsilon_5$};
    \draw (0.431564069,-0.028) node {\tiny $\upsilon_6$};
    \draw (0.397061018,-0.028) node {\tiny $\upsilon_7$};
    \draw (0.353553391,-0.028) node {\tiny $\upsilon_8$};
    \draw (0.000000000,-0.028) node {\tiny $\upsilon_9$};
    
    \draw (0.707106781,0.475) node {$P'_2$};
    \draw (0.707106781+0.3,0.2) node {$P'_3$};
    \draw (0.707106781-0.3,0.2) node {$P'_1$};
  \end{tikzpicture} 
\]
\endgroup
\caption{The Farey graph $\cF_4$ is the skeleton of a tiling of
$\overline{ \BH }$ by ideal $4$-angles like $P'_1$, $P'_2$, $P'_3$.
\smallskip \newline The closed path (blue), including the vertex $\upsilon_0 = \infty$ not shown in the figure, is provided by Lemma~\ref{lem:SW}(i).  It encloses a $4$-angulated $10$-gon $P'$ with vertices $\upsilon_0, \ldots, \upsilon_9$.  It takes the $q_{ \alpha }$th right turn at vertex $\upsilon_{ \alpha }$, hence $q_{ \alpha }$ equals the number of $4$-angles inside the path which are incident with $\upsilon_{ \alpha }$.  \smallskip \newline For instance, we have $q_4 = 3$ and $P'_1$, $P'_2$, $P'_3$ are incident with $\upsilon_4$.}
\label{fig:Farey}
\end{figure}

\begin{Lemma}
\label{lem:SW}
Assume that $\xi_0 \cdots \xi_{ n+2 }( \infty ) = \infty$.
\begin{enumerate}
\setlength\itemsep{4pt}

  \item  The vertices $\upsilon_0, \ldots, \upsilon_{ n+2 }$ form a closed path in $\cF_p$.
  
  \item  Let $P'$ be the full subgraph of $\cF_p$ defined by $\upsilon_0, \ldots, \upsilon_{ n+2 }$.  Then $P'$ is an $( n+3 )$-gon with a $p$-angulation which divides $P'$ into $p$-gons $P'_1, \ldots, P'_s$ such that
\begin{equation}
\label{equ:count}
   q_{ \alpha } = \big|\{\, P'_i \,|\, \mbox{$P'_i$ is incident with $\upsilon_{ \alpha }$} \,\}\big|
\end{equation}
for $0 \leqslant \alpha \leqslant n+2$, where the bars denote cardinality.

\end{enumerate}
\end{Lemma}

\begin{proof}
(i)  The vertices $\infty$ and $0$ are neighbours in $\cF_p$ because they are linked by the edge $L$.  Hence $\xi_0( \infty ) = \upsilon_1$ and $\xi_0( 0 ) = \infty = \upsilon_0$ show that $\upsilon_1$ is a neighbour of $\upsilon_0$.  For $2 \leqslant \alpha \leqslant n+2$ we have $\xi_0 \cdots \xi_{ \alpha-1 }( \infty ) = \upsilon_{ \alpha }$ and $\xi_0 \cdots \xi_{ \alpha-1 }( 0 ) = \xi_0 \cdots \xi_{ \alpha-2 }( \infty ) = \upsilon_{ \alpha-1 }$, so $\upsilon_{ \alpha }$ is a neighbour of $\upsilon_{ \alpha-1 }$.  Finally, $\xi_0 \cdots \xi_{ n+2 }( \infty ) = \infty = \upsilon_0$ and $\xi_0 \cdots \xi_{ n+2 }( 0 ) = \xi_0 \cdots \xi_{ n+1 }( \infty ) = \upsilon_{ n+2 }$, so $\upsilon_0$ is a neighbour of $\upsilon_{ n+2 }$.

(ii)  Since $\cF_p$ is the skeleton of a tiling of $\overline{ \BH }$
by ideal $p$-angles, the closed path $\upsilon_0, \ldots, \upsilon_{
n+2 }$ encloses a $p$-angulation of an $( n+3 )$-gon $P'$.  Now
observe that \cite[sec.\ 2]{SW} provides the following description of
$q_{ \alpha }$: In the path $\upsilon_0, \ldots, \upsilon_{ n+2 }$, the edge $\{ \upsilon_{ \alpha },\upsilon_{ \alpha+1 } \}$ is the $q_{ \alpha }$th right turn in $\cF_p$ relative to $\{ \upsilon_{ \alpha-1 },\upsilon_{ \alpha } \}$, where $\alpha$ is considered modulo $n+3$.  This implies Equation~\eqref{equ:count}; see Figure~\ref{fig:Farey} for an example.
\end{proof}

\section{Proof of Theorem A}
\label{sec:ThmA}

We need the following well-known result used implicitly in \cite[sec.\ 2]{Cu1} and \cite[sec.\ 2]{Cu2}.

\begin{Lemma}
\label{lem:C}
Let $P$ be an $( n+3 )$-gon with vertices $V = \{ 0,1, \ldots, n+2 \}$ where the numbering reflects the cyclic ordering of $V$.  Let $f$ be a frieze on $P$ and set
\[
  X_{ \alpha } = 
  \begin{pmatrix}
  	f( \alpha^-,\alpha^+ ) & -1 \\
  	1            &  0
  \end{pmatrix}
\]
for $\alpha \in V$.  Then the product is $X_0 \cdots X_{ n+2 } = -I$ where $I$ is the identity matrix.
\end{Lemma}

\begin{proof}
Suppose $\alpha \neq 0$.  Definition~\ref{def:frieze}, (ii) and (iv), imply
\[
  f( \alpha^-,\alpha^+ )f( \alpha,0 )
  = f( \alpha^-,\alpha )f( \alpha^+,0 )
    + f( \alpha^-,0 )f( \alpha,\alpha^+ )
  = f( \alpha^+,0 ) + f( \alpha^-,0 ).
\]
It follows that
\[
  f( \alpha^-,0 )
  = f( \alpha^-,\alpha^+ )f( \alpha,0 ) - f( \alpha^+,0 )
\]
and hence
\[
  X_{ \alpha }
  \begin{pmatrix}
  	f( \alpha,0 ) \\ f( \alpha^+,0 )
  \end{pmatrix}
  =
  \begin{pmatrix}
  	f( \alpha^-,0 ) \\ f( \alpha,0 )
  \end{pmatrix}.
\]
Iterating this equation gives
\[
  X_1 \cdots X_{ n+2 }
  \begin{pmatrix}
  	f( n+2,0 ) \\ f\big( (n+2)^+,0 \big)
  \end{pmatrix}
  =
  \begin{pmatrix}
  	f( 1^-,0 ) \\ f( 1,0 )
  \end{pmatrix}.
\]
By Definition~\ref{def:frieze}(i)-(iii), this reads
\[
  X_1 \cdots X_{ n+2 }
  \begin{pmatrix}
  	1 \\ 0
  \end{pmatrix}
  =
  \begin{pmatrix}
    0 \\ 1
  \end{pmatrix},
\]
implying
\[
  X_0 \cdots X_{ n+2 }
  \begin{pmatrix}
  	1 \\ 0
  \end{pmatrix}
  =
  X_0
  \begin{pmatrix}
    0 \\ 1
  \end{pmatrix}
  =
  \begin{pmatrix}
    -1 \\ 0
  \end{pmatrix}.
\]

A similar computation proves
\[
  X_0 \cdots X_{ n+2 }
  \begin{pmatrix}
  	0 \\ 1
  \end{pmatrix}
  =
  \begin{pmatrix}
    0 \\ -1
  \end{pmatrix},
\]
and concatenating the last two equations gives
\[
  X_0 \cdots X_{ n+2 }
  \begin{pmatrix}
    1 & 0 \\
  	0 & 1
  \end{pmatrix}
  =
  \begin{pmatrix}
    -1 &  0 \\
     0 & -1
  \end{pmatrix},
\]
proving the lemma.
\end{proof}

By virtue of Theorem~\ref{thm:frieze_patterns_and_friezes} and Lemma~\ref{lem:type_Lambda_q}, the following implies Theorem A.

\begin{Theorem}
\label{thm:A_prime}
Let $P$ be a polygon.  There is a bijection between $p$-angulations of $P$ and friezes of type $\Lambda_p$ on $P$.
\end{Theorem}

\begin{proof}
Definition~\ref{def:Phi} and Theorem~\ref{thm:B_prime} give an injection $\Phi$ from $p$-angulations of $P$ to friezes on $P$, which are of type $\Lambda_p$ by Remark~\ref{rmk:Phi2}.  We will show that it is surjective.

Let $P$ have vertices $V = \{ 0,\ldots,n+2 \}$ where the numbering reflects the cyclic ordering of $V$.  Let $f$ be a frieze of type $\Lambda_p$ on $P$.  For $\alpha \in V$ write
\begin{equation}
\label{equ:f_and_b}
  f( \alpha^-,\alpha^+ ) = q_{ \alpha }\lambda_p,
\end{equation}
and observe that $q_{ \alpha }$ is a positive integer.  Consider the matrices from Lemma~\ref{lem:C}:
\[
  X_{ \alpha }
  = 
  \begin{pmatrix}
  	f( \alpha^-,\alpha^+ ) & -1 \\
  	1            &  0
  \end{pmatrix}
  =  
  \begin{pmatrix}
  	q_{ \alpha }\lambda_p & -1 \\
  	1                     &  0
  \end{pmatrix}.
\]
Applying the homomorphism $\mu$ from Equation~\eqref{equ:mu} gives the  M\"{o}bius transformations $\xi_{ \alpha } = \mu( X_{ \alpha } )$ from Setup~\ref{set:Moebius} defined by
\[
  \xi_{ \alpha }( z ) = q_{ \alpha }\lambda_p - \frac{ 1 }{ z }.
\]
Lemma~\ref{lem:C} says $X_0 \cdots X_{ n+2 } = -I$ and applying $\mu$ gives $\xi_0 \cdots \xi_{ n+2 } = \id$.  In particular, $\xi_0 \cdots \xi_{ n+2 }( \infty ) = \infty$, so Lemma~\ref{lem:SW} provides a $p$-angulated $( n+3 )$-gon $P'$ with vertices $\upsilon_0, \ldots, \upsilon_{ n+2 }$.

By identifying $\upsilon_{ \alpha }$ with $\alpha$, the $p$-angulation of $P'$ becomes a $p$-angulation $D$ of $P$ which divides $P$ into $p$-gons $P_1, \ldots, P_s$.  By Equation~\eqref{equ:count} they satisfy
\begin{equation}
\label{equ:count2}
   q_{ \alpha} = \big|\{\, P_i \,|\, \mbox{$P_i$ is incident with $\alpha$} \,\}\big|
\end{equation}
for $\alpha \in V$.  

Set $f' = \Phi( D )$.  Remark~\ref{rmk:Phi2} shows
\[
  f'( \alpha^-,\alpha^+ )
  = \sum_{\mbox{\small $P_i$ is incident with $\alpha$}} \lambda_p
  = ( \ast )
\]  
for $\alpha \in V$.  Equation~\eqref{equ:count2} shows $( \ast ) = q_{ \alpha }\lambda_p$, whence Equation~\eqref{equ:f_and_b} implies $f' = f$ by Lemma~\ref{lem:agree}.
\end{proof}

\section{Questions}
\label{sec:questions}

\begin{enumerate}
\setlength\itemsep{4pt}

  \item  The frieze patterns of type $\Lambda_3$ are precisely the integral frieze patterns considered by Conway and Coxeter in \cite{CC1} and \cite{CC2}.  Is there a similarly useful characterisation of frieze patterns of type $\Lambda_p$ for $p \geqslant 4$?

  \item  Can the injection in Theorem B be promoted to a bijection?  In other words, is there a useful characterisation of the frieze patterns in the image?

  \item  Conway and Coxeter's theorem on triangulations and frieze patterns is closely related to the theory of cluster algebras of Dynkin type $A$.  How is Theorem A related to cluster algebras?

  \item  Conway and Coxeter's theorem on triangulations and frieze patterns is categorified by the theory of cluster characters on cluster categories of Dynkin type $A$.  Can Theorem A be categorified?

\end{enumerate}

\medskip
\noindent
{\bf Acknowledgement.}
We are deeply grateful to Ian Short who invited PJ to give a talk and gave him a copy of \cite{SW}, thereby providing the crucial impetus for this paper.

We thank Sophie Morier-Genoud, Yann Palu, and an anonymous referee for helpful comments on earlier versions.

This work was supported by EPSRC grant EP/P016014/1 ``Higher Dimensional Homological Algebra''.


\begin{thebibliography}{19}

\bibitem{ARS}  I.\ Assem, C.\ Reutenauer, and D.\ Smith, {\it Friezes}, Adv.\ Math.\ {\bf 225} (2010), 3134--3165.

\bibitem{BR}  F.\ Bergeron and C.\ Reutenauer, {\it $\operatorname{SL}_k$-tilings of the plane}, Illinois J.\ Math.\ {\bf 54} (2010), 263--300.

%

\bibitem{CC1} J.\ H.\ Conway and H.\ S.\ M.\ Coxeter, {\it
    Triangulated polygons and frieze patterns}, The Mathematical
  Gazette {\bf 57} (1973), 87--94.

\bibitem{CC2} J.\ H.\ Conway and H.\ S.\ M.\ Coxeter, {\it
    Triangulated polygons and frieze patterns} (continued from p.\
  94), The Mathematical Gazette {\bf 57} (1973), 175--183.

\bibitem{C}  H.\ S.\ M.\ Coxeter, {\it Frieze patterns}, Acta Arith.\ {\bf 18} (1971), 297--310.

\bibitem{C2}  H.\ S.\ M.\ Coxeter, ``Regular complex polytopes. Second edition'', Cambridge University Press, Cambridge, 1991.

\bibitem{CG}  H.\ S.\ M.\ Coxeter and S.\ L.\ Greitzer, ``Geometry revisited'', New Math.\ Library, Vol.\ 19, Random House, New York, 1967.

\bibitem{Cu1}  M.\ Cuntz, {\it Frieze patterns as root posets and affine triangulations}, European J.\ Combin.\ {\bf 42} (2014), 167--178.

\bibitem{Cu2}  M.\ Cuntz, On subsequences of quiddity cycles and Nichols algebras, preprint (2016).  {\tt math.CO/1610.02243}


\bibitem{HJ}  T.\ Holm and P.\ J\o rgensen, {\em $\SL2$-tilings and triangulations of the strip}, J.\ Combin.\ Theory Ser.\ A {\bf 120} (2013), 1817--1834.

\bibitem{L}  W.\ Lang, The field $\BQ ( 2\cos\big( \frac{ \pi }{ n } \big) )$, its Galois group, and length ratios in the regular $n$-gon, preprint (2012).  {\tt math.GR/1210.1018v1}

\bibitem{M}  S.\ Morier-Genoud, {\it Coxeter's frieze patterns at the crossroads of algebra, geometry and combinatorics}, Bull.\ London Math.\ Soc.\ {\bf 47} (2015), 895--938.

\bibitem{MOT}  S.\ Morier-Genoud, V.\ Ovsienko, and S.\ Tabachnikov, {\it $\SL2( \BZ )$-tilings of the torus, Coxeter--Conway friezes and Farey triangulations}, Enseign.\ Math.\ {\bf 61} (2015), 71--92.

\bibitem{SW}  I.\ Short and M.\ Walker, {\it Geodesic Rosen continued fractions}, Quart.\ J.\ Math.\ {\bf 67} (2016), 519--549.

\bibitem{S}  P.\ Steinbach, {\it Golden fields: A case for the heptagon}, Math.\ Mag.\ {\bf 70} (1997), 22--31.

\end{thebibliography}
\end{document}